\documentclass[12pt]{amsart}

\usepackage{mathrsfs,amssymb}

\newtheorem{theorem}{Theorem}[section]
\newtheorem{lemma}[theorem]{Lemma}

\theoremstyle{definition}

\theoremstyle{remark}
\newtheorem{remark}[theorem]{Remark}

\numberwithin{equation}{section}

\let \la=\lambda
\let \e=\varepsilon
\let \d=\delta

\let \a=\alpha
\let \f=\varphi
\let \b=\beta

\begin{document}

\title[On a counterexample]
{On a counterexample related to weighted weak type estimates for singular integrals}

\author{Marcela Caldarelli}
\address{Departamento de Matem\'atica\\
Universidad Nacional del Sur\\
Bah\'ia Blanca, 8000, Argentina}\email{marcela.caldarelli@uns.edu.ar}

\author{Andrei K. Lerner}
\address{Department of Mathematics,
Bar-Ilan University, 5290002 Ramat Gan, Israel}
\email{lernera@math.biu.ac.il}

\thanks{The second author was supported by the Israel Science Foundation (grant No. 953/13).}

\author{Sheldy Ombrosi}
\address{Departamento de Matem\'atica\\
Universidad Nacional del Sur\\
Bah\'ia Blanca, 8000, Argentina}\email{sombrosi@uns.edu.ar}

\begin{abstract}
We show that the Hilbert transform does not map $L^1(M_{\Phi}w)$ to $L^{1,\infty}(w)$ for every Young function $\Phi$ growing more slowly
than $t\log\log ({\rm e}^{\rm e}+t)$. Our proof is based on a construction of M.C.~Reguera and C. Thiele.
\end{abstract}



\maketitle

\section{Introduction}
Let $H$ be the Hilbert transform. One of open questions in the one-weighted theory of singular integrals is about the optimal Young function $\Phi$ for which
the weak type inequality
\begin{equation}\label{wh}
w\{x\in {\mathbb R}:|Hf(x)|>\la\}\le \frac{c}{\la}\int_{{\mathbb R}}|f|M_{\Phi}w\,dx\quad(\la>0)
\end{equation}
holds for every weight (i.e., non-negative measurable function) $w$ and any $f\in L^1(M_{\Phi}w)$, where $M_{\Phi}$ is the Orlicz maximal operator defined by
$$M_{\Phi}f(x)=\sup_{I\ni x}\inf\left\{\la>0:\frac{1}{|I|}\int_I\Phi\left(\frac{|f(x)|}{{\la}}\right)dx\le 1\right\}.$$
If $\Phi(t)=t$, then $M_{\Phi}=M$ is the standard Hardy-Littlewood maximal operator. If $\Phi(t)=t^r, r>1$, denote $M_{\Phi}f=M_rf$.

C. Fefferman and E.M. Stein \cite{FS} showed that if $H$ is replaced by the maximal operator $M$, then the corresponding inequality holds with $\Phi(t)=t$.
Next, A. C\'ordoba and C. Fefferman \cite{CF} proved (\ref{wh}) with $\Phi(t)=t^r, r>1$. This result was improved by C. P\'erez \cite{P} who showed that (\ref{wh}) holds
with $\Phi(t)=t\log^{\e}(\rm{e}+t),\e>0$ (see also \cite{HP} for a different proof of this result).

Very recently, C. Domingo-Salazar, M.T. Lacey and  G. Rey \cite{DLR}
obtained a further improvement; their result states that (\ref{wh}) holds whenever $\Phi$ satisfies
$$\int_{1}^{\infty}\frac{\Phi^{-1}(t)}{t^2\log(\rm{e}+t)}dt<\infty.$$
For example, one can take $\Phi(t)=t\log\log^{\a}({\rm e}^{\rm e}+t),\a>1$ or
$$\Phi(t)=t\log\log({\rm e}^{\rm e}+t)\log\log\log^{\a}({\rm e}^{{\rm e}^{\rm e}}+t)\quad(\a>1)$$
etc.

A question whether (\ref{wh}) is true with $\Phi(t)=t$ has become known as the Muckenhoupt-Wheeden conjecture. This conjecture was disproved
by M.C. Reguera and C. Thiele \cite{RT} (see also \cite{R} and \cite{CS} for  dyadic and multidimensional versions of this result).

Denote $\Psi(t)=t\log\log ({\rm e}^{\rm e}+t)$. It was conjectured in \cite{HP} that (\ref{wh}) holds with $\Phi=\Psi$.
The above mentioned result in \cite{DLR} establishes (\ref{wh}) for essentially every $\Phi$ growing faster than $\Psi$.

The main result of this note is the observation that the Reguera-Thiele example \cite{RT} actually shows that (\ref{wh}) does not
hold for every $\Phi$ growing more slowly than $\Psi$.

\begin{theorem}\label{main} Let $\Phi$ be a Young function such that
$$\lim_{t\to \infty}\frac{\Phi(t)}{t\log\log ({\rm e}^{\rm e}+t)}=0.$$
Then for every $c>0$, there exist $f,w$ and $\la>0$ such that
$$w\{x\in {\mathbb R}:|Hf(x)|>\la\}>\frac{c}{\la}\int_{{\mathbb R}}|f|M_{\Phi}w\,dx.$$
\end{theorem}

This theorem along with the main result in \cite{DLR} emphasizes that the case of $\Phi=\Psi$ is critical for (\ref{wh}). However, the question whether
(\ref{wh}) holds with $\Phi=\Psi$ remains open.

We mention briefly the main ideas of the Reguera-Thiele example~\cite{RT} and, in parallel, our novel points. 
First, it was shown in \cite{RT} that given $k\in {\mathbb N}$ sufficiently large, there is a weight $w_k$ supported on $[0,1]$ satisfying $Hw_k\ge ckw_k$ and $Mw_k\le cw_k$ on
some subset $E\subset [0,1]$. In Section~2, we show that the latter ``$A_1$ property" can be slightly improved until $M_rw_k\le cw_k$ with $r>1$ depending on $k$.
The second ingredient in \cite{RT} was the extrapolation argument of D. Cruz-Uribe and C. P\'erez~\cite{CUP}. This argument says that assuming (\ref{wh}) with $Mw$ on the right-hand
side, one can deduce a certain weighted $L^2$ inequality for~$H$. It is not clear how to extrapolate in a similar way starting with a general Orlicz maximal function $M_{\Phi}$ in (\ref{wh}).
In Section 3, we obtain a substitute of the argument in \cite{CUP} for $M_rw, r>1,$ instead of $Mw$.

\section{The Reguera-Thiele construction}
We describe below the main parts of the example constructed by M.C. Reguera and C. Thiele \cite{RT}.

An interval $I$ of the form $[3^jn,3^j(n+1)), j,n\in {\mathbb Z}$, is called a triadic interval.

Fix $k\in {\mathbb N}$ large enough. Given a triadic interval $I\subset [0,1)$, denote $I^{\Delta}=\frac{1}{3}I$, namely, $I^{\Delta}$ is the interval with the same center as $I$
and of one third of its length; further, denote by $P(I)$ a triadic interval adjacent to $I^{\Delta}$ and such that $|P(I)|=\frac{1}{3^k}|I|$. Observe that $P(I)$ can be situated either on
the left or on the right of $I^{\Delta}$, and we will return to this point a bit later.

Set now $J^{1}=[0,1)$ and $I_{1,1}=P(J^{1})$. Next, we subdivide $(J^{1})^{\Delta}$ into
$3^{k-1}$ triadic intervals of equal length, and denote them by $J_{m}^{2}, m=1,2,\dots,3^{k-1}.$
Set correspondingly $I_{2,m}=P(J_{m}^{2})$. Notice that $|J_{m}^{2}|=\frac{1}{3^{k}}$
and $|I_{2,m}|=\frac{1}{3^{2k}}$ for $m=1,2,...3^{k-1}$. Observe also that the intervals
$I_{1,1}$ and $I_{2,m}$ are pairwise disjoint.

Proceeding by induction, at $l$-th stage, we subdivide each interval $(J_{m}^{l-1})^{\Delta}$
into $3^{k-1}$ triadic intervals of equal length, and denote all obtained intervals by
$J_{m}^{l}, m=1,2,\dots,3^{(k-1)(l-1)}$. Set $I_{l,m}=P(J_{m}^{l})$. Then
$|J_{m}^l|=\frac{1}{3^{(l-1)k}}$ and $|I_{l,m}|=\frac{1}{3^{lk}}$, and the intervals
$\{I_{l,m}\}$ are pairwise disjoint.

Denote by ${\mathcal I}_l$ and ${\mathcal J}_l$ the families of all intervals $\{I_{l,m}\}$ and $\{J_m^l\}$, respectively,
and set $\Omega_l=\cup_{I\in {\mathcal I}_l}I$. Define the weight $w_k$ such that $w_k([0,1])=1$, $w_k$ is a constant on $\Omega_l$,
and for every $I\in {\mathcal I}_l$ and $J\in {\mathcal J}_{l+1}$, $w_k(I)=w_k(J)$ (we use the standard notation $w_k(E)=\int_Ew_k$).

It was proved in \cite{RT} that one can specify the situation of the intervals $\{I_{l,m}\}$ such that if $k>3000$ and $x\in \cup_{l,m}I_{l,m}^{\Delta}$, then
\begin{equation}\label{estH}
|Hw_k(x)|\ge (k/3)w_k(x);
\end{equation}
moreover,
$$
Mw_k(x)\le 7w_k(x) \quad(x\in \cup_{l,m}I_{l,m}^{\Delta}),
$$
irrespective of the precise configuration of $\{I_{l,m}\}$.

We will show that the latter estimate can be improved by means of replacing $Mw_k$ on the left-hand side by a larger operator $M_rw_k$ with $r>1$ depending on $k$.
In order to do that, we need a more constructive description of $w_k$.

\begin{lemma}\label{descr} We have,
\begin{equation}\label{wk}
w_{k}(x)=\sum_{l=1}^{\infty} \left( \frac{3^{k}}{3^{k-1}+1}\right) ^{l} \chi_{\Omega_l}(x).
\end{equation}
\end{lemma}

\begin{proof}
Assume that $w_k=\a_l$ on $\Omega_l$. Let $J\in {\mathcal J}_l$ and take $I\in {\mathcal I}_l$
such that $I\subset J$. Then
\begin{equation}\label{int}
w_k(J)=w_k(I)+w_k(J^{\Delta})=w_k(I)+\sum_{J'\in {\mathcal J}_{l+1}:J'\subset J^{\Delta}}w_k(J').
\end{equation}

Let $I'\in {\mathcal I}_{l-1}$. Then
$$w_k(J)=w_k(I')=\a_{l-1}|I'|=\a_{l-1}|J|.$$
Similarly, $w_k(J')=\a_l|J'|$, and also $w_k(I)=\a_{l}|I|=\a_l\frac{|J|}{3^k}$. Hence, (\ref{int}) implies
$$\a_{l-1}|J|=\a_l\frac{|J|}{3^k}+\a_l\sum_{J'\in {\mathcal J}_{l+1}:J'\subset J^{\Delta}}|J'|=\a_l\frac{|J|}{3^k}+\a_l\frac{|J|}{3}.$$
From this, $\a_l=\frac{3^k}{3^{k-1}+1}\a_{l-1}$, and therefore $\a_{l}=\Big(\frac{3^k}{3^{k-1}+1}\Big)^{l}\gamma$ for some $\gamma>0$.

From the condition $w_k([0,1])=1$, we obtain
\begin{eqnarray*}
1&=&w_k([0,1])=\gamma\sum_{l=1}^{\infty}\Big(\frac{3^k}{3^{k-1}+1}\Big)^{l}|\Omega_l|\\
&=&\gamma\sum_{l=1}^{\infty}\Big(\frac{3^k}{3^{k-1}+1}\Big)^{l}\frac{3^{(k-1)(l-1)}}{3^{kl}}=\gamma\frac{1}{3^{k-1}}\sum_{l=1}^{\infty}
\Big(\frac{3^{k-1}}{3^{k-1}+1}\Big)^{l}=\gamma,
\end{eqnarray*}
and therefore the lemma is proved.
\end{proof}

\begin{lemma}\label{Mr} Let $r=1+\frac{1}{3^{k+1}}$. Then for every $I\in {\mathcal I}_l, l\in {\mathbb N},$ and for all $x\in I^{\Delta}$,
$$M_rw_k(x)\le 21w_k(x).$$
\end{lemma}

\begin{proof} Let $I\in {\mathcal I}_l$, and let $x\in I^{\Delta}$. Take an arbitrary interval $R$ containing $x$. If $R\subset I$, then
$$
\left(\frac{1}{|R|}\int_Rw_k^r(y)dy\right)^{1/r}=\left(\frac{3^{k}}{3^{k-1}+1}\right)^{l}=w_k(x).
$$

Assume that $R\not\subset I$. Then $|R|\ge |I|/3$. Denote by ${\mathcal F}$ the family of all triadic intervals $I'\subset [0,1)$ such that $|I'|=|I|$ and $I'\cap R\not=\emptyset$. There are at most two intervals $I'\in {\mathcal F}$ not containing in $R$, and therefore,
\begin{equation}\label{sum}
\sum_{I'\in{\mathcal F}}|I'|\le |R|+\sum_{I'\in{\mathcal F}:I'\not\subset R}|I'|\le |R|+2|I|\le 7|R|.
\end{equation}

We claim that if $r=1+\frac{1}{3^{k+1}}$, then for every $I'\in {\mathcal F}$,
\begin{equation}\label{claim}
\left(\frac{1}{|I'|}\int_{I'}w_k^r(y)dy\right)^{1/r}\le 3w_k(x).
\end{equation}
This property would conclude the proof since then, by (\ref{sum}),
$$
\frac{1}{|R|}\int_{R}w_k^r(y)dy\le \sum_{I'\in {\mathcal F}}\frac{|I'|}{|R|}\frac{1}{|I'|}\int_{I'}w_k^r(y)dy\le 7(3w_k(x))^r.
$$

To show (\ref{claim}), one can assume that $I'$ has a non-empty intersection with the support of $w_k$. If $I'\not=J$ for some $J\in {\mathcal J}_{l+1}$, then
$I'\subset L$, where $L\in {\mathcal I}_{\nu}, \nu\le l$, and hence
$$\left(\frac{1}{|I'|}\int_{I'}w_k^r(y)dy\right)^{1/r}=\left(\frac{3^{k}}{3^{k-1}+1}\right)^{\nu}\le w_k(x).$$

It remains to consider the case when $I'=J$ for some $J\in {\mathcal J}_{l+1}$. Using that
for every $j\ge l+1$,  $J\in {\mathcal J}_{l+1}$ contains $3^{(k-1)(j-l-1)}$ intervals $I\in {\mathcal I}_j$, we obtain
\begin{eqnarray*}
\frac{1}{|I'|}\int_{I'}w_k^r(y)dy&=&3^{lk}\sum^{\infty}_{j=l+1}\sum_{I\in {\mathcal I}_j:I\subset I'}\int_{I}w_k^r(y)dy\\
&=&\sum^{\infty}_{j=l+1} 3^{(k-1)(j-l-1)} 3^{(l-j)k} \left( \frac{3^{k}}{3^{k-1}+1}\right)^{j r}\\
&=&\frac{1}{3^{k-1}}\sum_{j=1}^{\infty}3^{-j}\left( \frac{3^{k}}{3^{k-1}+1}\right)^{(j+l)r}.
\end{eqnarray*}
Therefore,
\begin{eqnarray*}
\frac{1}{|I'|}\int_{I'}w_k^r(y)dy&=&\frac{1}{3^{k-1}}\left(\sum_{j=1}^{\infty}3^{-j}\left( \frac{3^{k}}{3^{k-1}+1}\right)^{jr}\right)w_k(x)^r\\
&\le&\frac{1}{3^{k-1}}\frac{3}{3-(3^k/(3^{k-1}+1))^r}w_k(x)^r,
\end{eqnarray*}
whenever $\Big(\frac{3^k}{3^{k-1}+1}\Big)^r<3$.

If $r=1+\frac{1}{3^{k+1}}$, then
\begin{eqnarray*}
\Big(\frac{3^k}{3^{k-1}+1}\Big)^{1+\frac{1}{3^{k+1}}}\le 3^{\frac{1}{3^{k+1}}}\frac{3^k}{3^{k-1}+1}&\le& \Big(1+\frac{1}{3^k}\Big)\frac{3^k}{3^{k-1}+1}\\
&=&3-\frac{2}{3^{k-1}+1},
\end{eqnarray*}
Hence
$$
\frac{1}{|I'|}\int_{I'}w_k^r(y)dy\le \frac{3}{2}\frac{3^{k-1}+1}{3^{k-1}}w_k(x)^r\le 3w_k(x)^r,
$$
which completes the proof.
\end{proof}

\section{Extrapolation} Here we follow the extrapolation argument of D. Cruz-Uribe and C.~P\'erez \cite{CUP}, with some modifications.
\begin{lemma}\label{ext}
Assume that for every weight $w$ and for all $f\in L^1(M_rw)$,
$$\|Hf\|_{L^{1,\infty}(w)}\le A_r\|f\|_{L^1(M_r(w))}\quad (1<r<2).$$
Let $\a_r=\frac{r}{2-r}$.
There is $c>0$ such that for any weight $w$ supported in $[0,1]$ one has
$$\int_0^1\left(\frac{|Hw|}{(M_{\a_r}w)^{\a_r/r}}\right)^2w^{\a_r}dx\le cA_r^2\int_0^1wdx
\quad(1<r<2).$$
\end{lemma}

\begin{proof} Denote $\b_r=\frac{r(r-1)}{2-r}$. The numbers $\a_r$ and $\b_r$ are chosen in such a way in order to satisfy $\a_r-\b_r=r$ and $\a_r-\frac{2\b_r}{r}=1$.

Let $g\ge 0$. Since
$$\frac{1}{|I|}\int_I(gw)^r=\left(\frac{1}{w^{\a_r}(I)}\int_I(g^r/w^{\b_r})w^{\a_r}\right)\frac{w^{\a_r}(I)}{|I|},$$
we get
\begin{equation}\label{mr}
M_r(gw)(x)\le 2\Big(M^c_{w^{\a_r}}(g^r/w^{\b_r})(x)M_{\a_r}(w)(x)^{\a_r}\Big)^{1/r},
\end{equation}
where $M^c_v$ means the centered weighted maximal operator with respect to a weight $v$.

Using the initial assumption on $H$ along with (\ref{mr}), and applying H\"older's inequality along with the boundedness of $M^c_v$ on $L^p(v), p>1,$ we obtain
\begin{eqnarray*}
&&\int_{\{|Hf|>1\}}gw\le A_r \|f\|_{L^1(M_r(gw))}\\
&&\le 2A_r\int_{{\mathbb R}}\Big(|f|M_{\a_r}(w)^{\frac{\a_r}{r}}\frac{1}{w^{\a_r/2}}\Big)\Big(M^c_{w^{\a_r}}(g^r/w^{\b_r})^{\frac{1}{r}}w^{\a_r/2}\Big)dx\\
&&\le 2A_r\|f\|_{L^2\big((M_{\a_r}w)^{\frac{2\a_r}{r}}/w^{\a_r}\big)}\|M^c_{w^{\a_r}}(g^r/w^{\b_r})^{\frac{1}{r}}\|_{L^2(w^{\a_r})}\\
&&\le cA_r\|f\|_{L^2\big((M_{\a_r}w)^{\frac{2\a_r}{r}}/w^{\a_r}\big)}\|g\|_{L^2(w)}.
\end{eqnarray*}
Taking here the supremum over all $g\ge 0$ with $\|g\|_{L^2(w)}=1$ yields
$$\|Hf\|_{L^{2,\infty}(w)}\le cA_r\|f\|_{L^2\big((M_{\a_r}w)^{\frac{2\a_r}{r}}/w^{\a_r}\big)}.$$

By duality, the latter inequality is equivalent to
$$\|Hf\|_{L^2\big(w^{\a_r}/(M_{\a_r}w)^{\frac{2\a_r}{r}}\big)}\le cA_r\|f/w\|_{L^{2,1}(w)},$$
where $L^{2,1}(w)$ is the weighted Lorentz space. It remains to take here $f=w$ and use that
$$\|\chi_{[0,1]}\|_{L^{2,1}(w)}=\int_0^{w([0,1])}t^{-1/2}dt=2w([0,1])^{1/2}.$$
\end{proof}

\section{Proof of Theorem \ref{main}}
Our goal is to use the extrapolation Lemma \ref{ext}, assuming (\ref{wh}) with a general Orlicz maximal function $M_{\Phi}$.
Hence, we need a relation between $M_{\Phi}$ and $M_r$ with possibly good dependence of the corresponding constant on $r$  when $r\to 1$.
Such a relation was recently obtained in \cite{DL} (see Lemma 6.2 and inequality (6.4) there). For the reader convenience we include a proof here.

\begin{lemma} For all $x\in {\mathbb R}$,
\begin{equation}\label{MPhi}
M_{\Phi}f(x)\le \left(2\sup_{t\ge \Phi^{-1}(1/2)}\frac{\Phi(t)}{t^r}\right)^{1/r}M_{r}f(x)\quad(r>1).
\end{equation}
\end{lemma}

\begin{proof} For any interval $I\subset {\mathbb R}$, 
\begin{eqnarray*}
\int_I\Phi\left(\frac{|f|}{\la}\right)
&=&\int\limits_{\{x\in I:|f|<\Phi^{-1}(1/2)\la\}}\Phi\left(\frac{|f|}{\la}\right)
+\int\limits_{\{x\in I:|f|\ge\Phi^{-1}(1/2)\la\}}\Phi\left(\frac{|f|}{\la}\right)\\
&\le& \frac{|I|}{2}+c_r\int_I(|f|/\la)^rdx,
\end{eqnarray*}
where $c_r=\sup_{t\ge \Phi^{-1}(1/2)}\frac{\Phi(t)}{t^r}$. Therefore, setting $\la_0=\Big(\frac{2c_r}{|I|}\int_I|f|^r\Big)^{1/r}$,
we obtain $\frac{1}{|I|}\int_I\Phi(|f|/\la_0)dx\le 1$, which proves (\ref{MPhi}).  
\end{proof}

It follows easily from (\ref{MPhi}) that 
\begin{equation}\label{mphi}
M_{\Phi}(x)\le c\left(\sup_{t\ge 1}\frac{\Phi(t)^{1/r}}{t}\right)M_rf(x)\quad (r>1),
\end{equation}
where $c$ may depend on $\Phi$ but it does not depend on $r$. 

\begin{proof}[Proof of Theorem \ref{main}]
Suppose, by contrary, that (\ref{wh}) holds. Then combining (\ref{mphi}) with Lemma \ref{ext}, we obtain
$$\int_0^1\left(\frac{|Hw|}{(M_{\a_r}w)^{\a_r/r}}\right)^2w^{\a_r}dx\le c\left(\sup_{t\ge 1}\frac{\Phi(t)^{1/r}}{t}\right)^2\int_0^1wdx
\quad(1<r<2).$$
Set here $r=r_k=1+\frac{1}{2\cdot 3^{k+1}+1}$, and $w=w_k$ as constructed in Section~2. Then $\a_{r_k}=\frac{r_k}{2-r_k}=1+\frac{1}{3^{k+1}}$.
Applying (\ref{estH}) along with Lemma \ref{Mr} yields
\begin{eqnarray*}
\int_0^1\left(\frac{|Hw_k|}{(M_{\a_r}w_k)^{\a_r/r}}\right)^2w_k^{\a_r}dx&\ge& \frac{k^2}{9\cdot 27^{\frac{2}{2-r_k}}}
\int_{\cup_{l\in{\mathbb N}}\cup_{I\in {\mathcal I}_l}I^{\Delta}}w_k\\
&=&\frac{k^2}{27^{1+\frac{2}{2-r_k}}}\int_0^1w_k,
\end{eqnarray*}
and we obtain
\begin{equation}\label{estphi}
k\le c\sup_{t\ge 1}\frac{\Phi(t)^{1/r_k}}{t}.
\end{equation}

It remains to estimate the right-hand side of (\ref{estphi}). Write $\Phi(t)=t\log\log({\rm e}^{\rm e}+t)\phi(t)$, where
$\lim_{t\to \infty}\phi(t)=0$. If $t>{\rm e}^{r'}$, then
$$\log\log t=\log(r')+\log\log t^{1/r'}\le \log(r')+t^{1/r'},$$
and hence
$$\frac{\Phi(t)^{1/r}}{t}=\frac{\big(\log\log({\rm e}^{\rm e}+t)\phi(t)\big)^{1/r}}{t^{1/r'}}\le c(\log r')^{1/r}(\sup_{t\ge {\rm e}^{r'}}\phi(t))^{1/r}.$$

On the other hand, if $0<\d<1$, then
\begin{eqnarray*}
\sup_{1\le t\le e^{r'}}\frac{\Phi(t)^{1/r}}{t}&\le& \sup_{1\le t\le {\rm e}^{{\rm e}^{(\log r')^{\d}}}}\big(\log\log({\rm e}^{\rm e}+t)\phi(t)\big)^{1/r}\\
&+&\sup_{{\rm e}^{{\rm e}^{(\log r')^{\d}}}\le t\le e^{r'}}\big(\log\log({\rm e}^{\rm e}+t)\phi(t)\big)^{1/r}\\
&\le& c\Big((\log r')^{\d/r}+(\log r')^{1/r}\sup_{t\ge {\rm e}^{{\rm e}^{(\log r')^{\d}}}}\phi(t)^{1/r}\Big).
\end{eqnarray*}

Setting $\displaystyle\b_k=\sup_{t\ge {\rm e}^{{\rm e}^{(\log r_k')^{\d}}}}\phi(t)^{1/r_k}$ and combining both cases, we obtain
\begin{eqnarray*}
\sup_{t\ge 1}\frac{\Phi(t)^{1/r_k}}{t}&\le& c((\log r_k')^{\d/r_k}+\b_k(\log r_k')^{1/r_k})\\
&\le& c(k^{\d}+\b_kk).
\end{eqnarray*}
Since $\b_k\to 0$ as $k\to \infty$, we arrive to a contradiction with (\ref{estphi}), and therefore the theorem is proved.
\end{proof}

\begin{remark}
The following inequality is contained implicitly in \cite{HP}:
$$
\la w\{x\in {\mathbb R}:|Hf(x)|>\la\}\le c\log(r')\|f\|_{L^1(M_rw)}\quad(r>1).
$$
The proof of Theorem \ref{main} shows that $\log(r')$ here is optimal, namely, it cannot be replaced by $\f(r')$ for any increasing $\f$ such that
$\displaystyle\lim_{t\to\infty}\frac{\f(t)}{\log t}=~0$.
\end{remark}

\end{document}